\documentclass[12pt]{amsart}
\usepackage{amsmath}
\usepackage{amssymb}
\usepackage{mathtools,nccmath,textcomp}
\usepackage{tikz}
\usepackage[all]{xy}
\usepackage{longtable}
\allowdisplaybreaks[1]
\newtheorem{theorem}{Theorem}
\newtheorem{lemma}[theorem]{Lemma}

\newtheorem{corollary}[theorem]{Corollary}

\theoremstyle{definition}

\newtheorem{example}{Example}

\theoremstyle{remark}
\newtheorem{remark}[theorem]{Remark}

\numberwithin{equation}{section}

\DeclareMathAlphabet{\matheur}{U}{eur}{m}{n}

\newcommand{\m}{\mathrm{m}}
\newcommand{\C}{\mathbb{C}}

\newcommand{\R}{\mathbb{R}}
\newcommand{\Q}{\mathbb{Q}}

\newcommand{\Z}{\mathbb{Z}}

\newcommand{\re}{\mathop{\mathrm{Re}}} 

\renewcommand\d{{\mathrm d}}

\newmuskip\pFqskip
\mathchardef\pFcomma=\mathcode`, 

\renewcommand{\d}{\mathrm d}

\begin{document}
\title[A functional identity for Mahler measures of non-tempered polynomials]{A functional identity for Mahler measures of non-tempered polynomials}

\author{Detchat Samart}
\address{Department of Mathematics, Faculty of Science, Burapha University, Chonburi, Thailand 20131} \email{petesamart@gmail.com}



\date{\today}

\maketitle

\begin{abstract}
We establish a functional identity for Mahler measures of the two-parametric family $P_{a,c}(x,y)=a(x+1/x)+y+1/y+c$. Our result extends an identity proven in a paper of Lal\'{i}n, Zudilin and Samart. As a by-product, we obtain evaluations of $\m(P_{a,c})$ for some algebraic values of $a$ and $c$ in terms of special values of $L$-functions and logarithms. We also give a sufficient condition for validity of a certain identity between the elliptic integrals of the first and the third kind, which implies several identities for $\m(P_{a,c})$.
\end{abstract}

\section{Background and motivation}\label{S:background}

The (logarithmic) Mahler measure of a nonzero Laurent polynomial $P\in \C[x_1^{\pm 1},\ldots,x_n^{\pm 1}]$ is defined by 
\begin{align*}
\m(P)&= \frac{1}{(2\pi i)^n} \idotsint\limits_{|x_1|=\cdots= |x_n|=1} \log |P(x_1,\ldots,x_n)|\frac{dx_1}{x_1}\cdots\frac{dx_n}{x_n}\\
&=\int_0^1 \cdots \int_0^1 \log |P(e^{2\pi i \theta_1},\ldots,e^{2\pi i \theta_n})| d\theta_1 \cdots d\theta_n.
\end{align*} 
In recent work \cite{LSZ}, Lal\'{i}n, Zudilin and the present author investigated Mahler measures of the family 
\begin{equation*}
P_{a,c}:=P_{a,c}(x,y)= a\left(x+\frac{1}{x}\right)+y+\frac{1}{y}+c,
\end{equation*}
where $a$ and $c$ are complex parameters with $a\ne 0$. The one-parametric family 
\begin{equation*}
P_{1,k}(x,y)= x+\frac{1}{x}+y+\frac{1}{y}+k
\end{equation*}
has long been studied and, for certain values of $k$, the Mahler measure of $P_{1,k}$ is known (at least numerically) to be a rational multiple of an $L$-value of the elliptic curve defined by the zero locus of $P_{1,k}$. More precisely, Boyd \cite[Tab.~1]{Boyd} and Rodriguez Villegas \cite[Tab.~4]{RV} discovered from their numerical computations that for several values of $k$ such that $k^2\in \mathbb{Z}$ 
\begin{equation}\label{E:Boyd}
\m(P_{1,k})\stackrel{?}= r_k L'(E_k,0),
\end{equation}
where $r_k\in \Q^\times$, $E_k$ is the minimal model of the curve \[A_k: y^2 =x^3+\frac{k^2}{8}\left(\frac{k^2}{8}-1\right)x^2+\frac{k^4}{256}x,\] and $A\stackrel{?}= B$ means the real numbers $A$ and $B$ agree to at least $25$ decimal places. In addition, they hypothesized that $1/r_k \in \Z$ for all sufficiently large $k$. Although a great amount of effort has been put in by many researchers to rigorously verify $\eqref{E:Boyd}$, only a small number of cases are known to be true. We give a comprehensive list of proven identities together with the values of $r_k$ in Table~\ref{Ta:Pk} below. (Note that [A]+[B] in the last column means that the formula(s) in the same row is a direct consequence of results in [A] and [B], while their proofs may not have been given explicitly.)
\begin{table}[ht]\label{Ta:Pk}
\centering \def\arraystretch{1.1}
    \begin{tabular}{ | c | c | c | c |}
    \hline
    $k$ & Conductor of $E_k$ & $r_k$ & Reference(s)\\ \hline
    $1$ & $15$ & $1$ & \cite{RZ15}, \cite{Zudilin}  \\ \hline
    $3i,5,16$ & $15$ & $5,6,11$ & \cite{Lalin10}$+$\cite{RZ15}  \\ \hline
    $i$ & $17$ & $2$ & \cite{Zudilin}  \\ \hline
    $3$ & $21$ & $2$ & \cite{Brunault}, \cite{LSZ}  \\ \hline
    $2$ & $24$ & $1$ & \cite{LR}, \cite{RZ}, \cite{Zudilin}  \\ \hline
    $3\sqrt{2}$ & $24$ & $\frac{5}{2}$ & \cite{RV}  \\ \hline
    $8$ & $24$ & $4$ & \cite{LR}, \cite{RZ}  \\ \hline
    $\sqrt{2}i$ & $24$ & $\frac{3}{2}$ & \cite{RV}+\cite{LR}  \\ \hline
    $4i,2\sqrt{2}$ & $32$ & $2,1$ & \cite{RV}  \\ \hline
    $2i$ & $40$ & $1$ & \cite{Mellit}, \cite{Zudilin}  \\ \hline
    $12$ & $48$ & $2$ & \cite{Brunault}  \\ \hline
    $\sqrt{2}$ & $56$ & $\frac{1}{4}$ & \cite{Zudilin}  \\ \hline
    $4\sqrt{2}$ & $64$ & $1$ & \cite{RV}  \\ \hline
    \end{tabular}
\caption{Proven formulas for $\m(P_{1,k})$}
\label{T2}
\end{table}

Let $P\in \C[x^{\pm 1},y^{\pm 1}]$ for which there exist $a,b\in \Z$ such that $x^ay^bP$ is a monic quadratic polynomial in variable $y$ with coefficients in $\C[x,1/x]$. Then we can write $x^ay^bP$ as
\[x^ay^bP = y^2+B(x)y+C(x)= (y-y_+(x))(y-y_-(x)),\]
where $y_{\pm}(x)=(-B(x)\pm \sqrt{B(x)^2-4C(x)})/2$. Using the above factorization, the notion of {\it half-Mahler measures} was introduced in \cite{LSZ}. They arise naturally as two quantities $\m^{\pm}(P)$ which make up the Mahler measure $\m(P)$, namely
\[\m^{\pm}(P)= \frac{1}{2\pi i}\int_{|x|=1}\log^+|y_{\pm}(x)|\frac{dx}{x}=\int_0^1 \log^+|y_{\pm}(e^{2\pi i\theta})|d\theta,\]
where $\log^+(r)= \max\{0,\log r\}$. The fact that $\m(P)= \m^+(P)+\m^-(P)$ follows directly from Jensen's formula. It was proven in \cite[Thm.~2]{LSZ} that for $0<k<4$ 
\begin{align} 
\m(P_{1,k})&= \m^{-}(P_{a,c})-3\m^{+}(P_{a,c}),\label{E:LSZ}\\
\m(P_{a,c})&= \log a, \label{E:log}
\end{align}
where \[a=\sqrt{(4+k)/(4-k)}, \quad c=k/\sqrt{4-k}.\] The authors of \cite{LSZ} also apply a special case of \eqref{E:LSZ} and \eqref{E:log} to obtain a rigorous proof of \eqref{E:Boyd} for $k=3$, with the aid of Ramanujan's modular equation in level $21$ and a formula for regulators of modular units due to Brunault, Mellit, and Zudilin \cite{Zudilin}. The main objective of this paper is to extend \eqref{E:LSZ} for $k>4$. For the rest of this section, we shall assume that $a$ and $c$ are given in terms of $k$ as above.
\begin{theorem}\label{T:main}
Let $k\in (4,\infty)$. Then we have
\begin{equation}\label{E:main}
\m(P_{1,k})=
 2(\m^{+}(P_{a,c})-\m^{-}(P_{a,c}))+\frac{1}{2}\log \left(\mfrac{k-4}{k+4}\right).
\end{equation}
\end{theorem}
We prove Theorem~\ref{T:main} in Section~\ref{S:proof}. In addition, we show that $\m^-(P_{a,c})=0$ for $k>2(1+\sqrt{5})$, so the following result follows easily.

\begin{corollary}\label{Cor:1}
If $k>2(1+\sqrt{5})$, then the following identity is true:
\begin{equation}\label{E:main2}
\m(P_{1,k})=2\m(P_{a,c})+\frac{1}{2}\log \left(\mfrac{k-4}{k+4}\right).
\end{equation}
\end{corollary} 
Note that, after desingularization, the curve $P_{a,c}=0$ is $2$-isogenous to the curve $P_{1,k}=0$ \cite[Sect.~5]{LSZ}, so their $L$-functions coincide. Applying Corollary~\ref{Cor:1} together with the known results in Table~\ref{Ta:Pk}, we immediately obtain the following interesting evaluation of $\m(P_{a,c})$ for some (complex) algebraic values of $a$ and $c$.
\begin{corollary}\label{Cor:main}
For $k\in \{4\sqrt{2},8,12,16\}$, the following formula holds:
\begin{equation}\label{E:nt}
\m(P_{a,c})= \frac{r_k}{2}L'(E_k,0)-\frac{1}{4}\log \left(\frac{k-4}{k+4}\right),
\end{equation}
where $r_k$ is as given in Table~\ref{Ta:Pk}.
\end{corollary}
\begin{remark}
By \eqref{E:main}, one can also write $\m^+(P_{a,c})-\m^-(P_{a,c})$ as a rational linear combination of $L'(E_k,0)$ and $\log((k-4)/(k+4))$ for $k=3\sqrt{2}$ and $k=5$. However, it is nonobvious to us whether each of $\m^{\pm}(P_{a,c})$ is expressible in terms of the $L$-value and other meaningful quantities, even from a numerical point of view. Indeed, we are unable to relate $\m(P_{a,c})$ to $L$-values in these cases.
\end{remark}

The family $P_{1,k}$ is {\it tempered} in the sense that the Mahler measure of the polynomial associated to each side of its Newton polygon vanishes \cite{RV}. On the other hand, if $|a|\neq 1$, then $P_{a,c}$ is non-tempered. Temperedness implies triviality of the tame symbols of elements in the second $K$-group of the corresponding elliptic curve, which gives rise to a formula like \eqref{E:Boyd}. It is usually harder to prove a formula like \eqref{E:nt} directly, compared to \eqref{E:Boyd}, as more delicate analysis is required when computing the regulator integral corresponding to a non-tempered polynomial \cite{LSZ,Giard,MS}.

\section{Proof of the main result}\label{S:proof}
In the proof of our main result, we employ a method initially used by Bertin and Zudilin \cite{BZ}. This method relies crucially on the fact that the derivatives of the Mahler (and half-Mahler) measures in \eqref{E:main} with respect to the real parameter $k$ can be written in terms of elliptic integrals, which happen to satisfy a special identity. The details of the proof are summarized in Lemma~\ref{L:dfk} and Lemma~\ref{L:main} below and a generalization of the identity between elliptic integrals is discussed in Appendix~\ref{A:1}.

Let us first introduce the following notation for the sake of brevity:
\begin{align*}
f(k)&=\m(P_{1,k}),\\
h(k)&= \m^+(P_{a,c})-\m^-(P_{a,c}),
\end{align*}
where $a=\sqrt{(4+k)/(4-k)}$ and $c=k/\sqrt{4-k}$. We also adopt the following notation for the complete elliptic integrals of the first, second, and third kind:
\begin{align*}
K(z)&=\int_0^1\frac{\d x}{\sqrt{(1-x^2)(1-z^2x^2)}}, \qquad
E(z)=\int_0^1\frac{\sqrt{1-z^2x^2}}{\sqrt{1-x^2}}\d x,
\\
\Pi(n,z)&=\int_0^1\frac{\d x}{(1-nx^2)\sqrt{(1-x^2)(1-z^2x^2)}}.
\end{align*}
\begin{lemma}\label{L:dfk}
Let $k\in (4,\infty)$. Then 
\[\dfrac{df(k)}{dk}= \frac{2}{k\pi}K\left(\frac{4}{k}\right).\]
\end{lemma}
\begin{proof}
The derivative of $\m(P_{1,k})$ with respect to the parameter $k$ is known to be expressible in terms of a $_2F_1$-hypergeometric function \cite{RV}, which can then be easily translated into the complete elliptic integral of the first kind. However, we give a direct proof of this identity below.

Consider 
\[yP_{1,k}= y^2+\left(x+\frac{1}{x}+k\right)y+1= (y-y_+(x))(y-y_-(x)),\]
where $y_{\pm}(x)= \frac{-B(x)\pm \sqrt{B(x)^2-4}}{2}$ and $B(x):=B_k(x)=x+1/x+k.$ Since $k>4$, we have that, for $\theta\in\R$, $B(e^{i\theta})= 2\cos \theta +k>2$, so $B(e^{i\theta})^2-4>0.$ Also, since $y_+(x)y_-(x)=1$, we have $|y_+(x)|<1<|y_-(x)|$ for $|x|=1$. Using the symmetry $y_-(x)=y_-(x^{-1})$, we deduce that 
\begin{align*}
f(k)=\m^{-}(P_{1,k})= \frac{1}{\pi} \re \int_0^\pi \log\left(\frac{B(e^{i\theta})+\sqrt{B(e^{i\theta})^2-4}}{2}\right)d\theta.
\end{align*}
It follows that 
\begin{align*}
\dfrac{df(k)}{dk}&= \frac{1}{\pi} \re \int_0^\pi \frac{1}{\sqrt{(2\cos \theta+k)^2-4}}d\theta\\
&=\frac{1}{\pi}\re \int_{-1}^1 \frac{1}{\sqrt{(2t+k-2)(2t+k+2)}}\frac{dt}{\sqrt{1-t^2}},
\end{align*}
where the change of variables $t=\cos\theta$ is used in the second equality. Finally, we apply the change of variables $t=\frac{(4-2k)x^2+k}{4x^2-k}$ in the last integral above to obtain 
\[\frac{df(k)}{dk}= \frac{2}{k\pi}\int_0^1 \frac{dx}{\sqrt{(1-x^2)\left(1-\frac{16}{k^2}x^2\right)}}=\frac{2}{k\pi}K\left(\frac{4}{k}\right).\]
\end{proof}

\begin{lemma}\label{L:main}
Let $k\in (4,\infty)$. Then the following identity is true:
\begin{equation}\label{E:dfk2}
\dfrac{df(k)}{dk}=
 2\dfrac{dh(k)}{dk}+\mfrac{4}{k^2-16}.
\end{equation}
\end{lemma}
\begin{proof}
Let \[\tilde{P}_k:=\tilde{P}_k(x,y)=-iP_{a,c}(x,iy)=\sqrt{\frac{k+4}{k-4}}\left(x+\frac{1}{x}\right)+y-\frac{1}{y}-\frac{k}{\sqrt{k-4}}.\]
Then we have
\begin{align*}
y\tilde{P}_k= y^2+B(x)y-1=(y-y_+(x))(y-y_-(x)),
\end{align*}
where 
\begin{align*}
B(x)&:=B_k(x)=\sqrt{\frac{k+4}{k-4}}\left(x+\frac{1}{x}\right)-\frac{k}{\sqrt{k-4}},\\
y_{\pm}(x)&= \frac{-B(x)\pm \sqrt{B(x)^2+4}}{2}.
\end{align*}
It can be seen from the definition of half-Mahler measures that for any real $k>4$
\begin{equation*}
\m^{\pm}(P_{a,c})=\m^{\pm}(\tilde{P}_k),
\end{equation*}
so $h(k)=\m^+(\tilde{P}_k)-\m^-(\tilde{P}_k).$
Note that for $\theta\in \R$
\[|y_{\pm}(e^{i\theta})|=\frac{1}{2\sqrt{k-4}}\left|(2\sqrt{k+4}\cos \theta-k)\mp \sqrt{(2\sqrt{k+4}\cos \theta-k)^2+4(k-4)}\right|.\]
If $\cos\theta > \frac{k}{2\sqrt{k+4}}$, then  $|y_+(e^{i\theta})|<1<|y_-(e^{i\theta})|$ and  if $\cos\theta < \frac{k}{2\sqrt{k+4}}$, then  $|y_-(e^{i\theta})|<1<|y_+(e^{i\theta})|.$ Therefore, by the change of variable $t=\cos \theta$, we have 
\begin{align*}
\m^+(\tilde{P}_k)&= -\frac{1}{\pi}\re\int_{-1}^{\frac{k}{2\sqrt{k+4}}}\log(\tilde{B}+\sqrt{\tilde{B}^2+1})\frac{dt}{\sqrt{1-t^2}},\\
\m^-(\tilde{P}_k)&= -\frac{1}{\pi}\re\int_{\frac{k}{2\sqrt{k+4}}}^{1}\log(\tilde{B}-\sqrt{\tilde{B}^2+1})\frac{dt}{\sqrt{1-t^2}},
\end{align*}
where $\tilde{B}= \frac{2t\sqrt{k+4}-k}{2\sqrt{k-4}}.$ Simple calculations yield 
\begin{align*}
\frac{\partial}{\partial k}\log(\tilde{B}+\sqrt{\tilde{B}^2+1})&=\frac{1}{\sqrt{\tilde{B}^2+1}}\frac{\partial \tilde{B}}{\partial k}=-\frac{4}{k^2-16}\frac{t+\frac{(k-8)\sqrt{k+4}}{16}}{\sqrt{t^2-\frac{k}{\sqrt{k+4}}t+\frac{k^2+4k-16}{4(k+4)}}},\\
\frac{\partial}{\partial k}\log(\tilde{B}-\sqrt{\tilde{B}^2+1})&=-\frac{1}{\sqrt{\tilde{B}^2+1}}\frac{\partial \tilde{B}}{\partial k}=\frac{4}{k^2-16}\frac{t+\frac{(k-8)\sqrt{k+4}}{16}}{\sqrt{t^2-\frac{k}{\sqrt{k+4}}t+\frac{k^2+4k-16}{4(k+4)}}}.
\end{align*}
Then we apply Leibniz's rule and the substitution $t\mapsto -t$ to obtain
\begin{align*}
\dfrac{d}{dk}\m^+(\tilde{P}_k)&= -\frac{4}{k^2-16}\re \int_{-\frac{k}{2\sqrt{k+4}}}^1\frac{t-\frac{(k-8)\sqrt{k+4}}{16}}{\sqrt{t^2+\frac{k}{\sqrt{k+4}}t+\frac{k^2+4k-16}{4(k+4)}}}\frac{dt}{\sqrt{1-t^2}},\\
\dfrac{d}{dk}\m^-(\tilde{P}_k)&= \frac{4}{k^2-16}\re \int_{-1}^{-\frac{k}{2\sqrt{k+4}}}\frac{t-\frac{(k-8)\sqrt{k+4}}{16}}{\sqrt{t^2+\frac{k}{\sqrt{k+4}}t+\frac{k^2+4k-16}{4(k+4)}}}\frac{dt}{\sqrt{1-t^2}}.
\end{align*}
Hence it follows that 
\[\dfrac{dh(k)}{dk}=\dfrac{d}{dk}\left(\m^+(\tilde{P}_k)-\m^-(\tilde{P}_k)\right)=-\frac{4}{k^2-16}\re \int_{-1}^1\frac{t-\frac{(k-8)\sqrt{k+4}}{16}}{\sqrt{t^2+\frac{k}{\sqrt{k+4}}t+\frac{k^2+4k-16}{4(k+4)}}}\frac{dt}{\sqrt{1-t^2}}.\]
Next, we transform the above integral into complete elliptic integrals using a standard procedure (see, for example, \cite[Ch.~3]{Hall}). In summary, we let 
\begin{align*}
\alpha = -\frac{\sqrt{k+4}}{2},\quad \beta = -\frac{2}{\sqrt{k+4}},\quad t=\frac{\alpha x-\beta}{x-1}.
\end{align*}
Then we have 
\begin{multline} \label{E:or}
\int_{-1}^1\frac{t-\frac{(k-8)\sqrt{k+4}}{16}}{\sqrt{t^2+\frac{k}{\sqrt{k+4}}t+\frac{k^2+4k-16}{4(k+4)}}}\frac{dt}{\sqrt{1-t^2}}\\=\int_{-\frac{2}{\sqrt{k+4}}}^{\frac{2}{\sqrt{k+4}}}\left(-\frac{k+4}{8}+\frac{1}{1-x^2}+\frac{x}{1-x^2}\right)\frac{dx}{\sqrt{(B_1x^2+A_1)(B_2x^2+A_2)}},
\end{multline}
where $A_1=B_2=4/k, A_2= (k-4)/k,$ and $B_1=-(k+4)/k.$ Applying the change of variables $u=x^2$ results in
\begin{equation}\label{E:1}
\int_{-\frac{2}{\sqrt{k+4}}}^{\frac{2}{\sqrt{k+4}}}\frac{x}{1-x^2}\frac{dx}{\sqrt{(B_1x^2+A_1)(B_2x^2+A_2)}}=0. 
\end{equation}
Next, we use the substitution $x\mapsto \sqrt{-A_1/B_1}x$ to deduce 
\begin{align*}
\int_{0}^{\frac{2}{\sqrt{k+4}}}\frac{dx}{\sqrt{(B_1x^2+A_1)(B_2x^2+A_2)}}&=\sqrt{-\frac{A_1}{B_1}}\int_0^1 \frac{dx}{\sqrt{(-A_1x^2+A_1)\left(-\frac{A_1B_2}{B_1}x^2+A_2\right)}}\\
&=\sqrt{-\frac{1}{A_2B_1}}\int_0^1 \frac{dx}{\sqrt{(1-x^2)\left(1-\frac{A_1B_2}{A_2B_1}x^2\right)}}\\
&= \frac{k}{\sqrt{k^2-16}}K\left(\sqrt{-\frac{16}{k^2-16}}\right).
\end{align*}
Therefore, we have 
\begin{equation}\label{E:2}
\int_{-\frac{2}{\sqrt{k+4}}}^{\frac{2}{\sqrt{k+4}}}\frac{dx}{\sqrt{(B_1x^2+A_1)(B_2x^2+A_2)}}=\frac{2k}{\sqrt{k^2-16}}K\left(\sqrt{-\frac{16}{k^2-16}}\right).
\end{equation}
With the same substitution, it can be shown that
\begin{equation}\label{E:3}
\int_{-\frac{2}{\sqrt{k+4}}}^{\frac{2}{\sqrt{k+4}}}\frac{1}{1-x^2}\frac{dx}{\sqrt{(B_1x^2+A_1)(B_2x^2+A_2)}}=\frac{2k}{\sqrt{k^2-16}}\Pi\left(\frac{4}{k+4},\sqrt{-\frac{16}{k^2-16}}\right).
\end{equation}
Then we substitute \eqref{E:1}-\eqref{E:3} into \eqref{E:or} and apply the identities \cite[Eq.~15.8.1, 16.16.8]{DLMF}
\begin{equation*}
K(\sqrt{z})=\frac{1}{\sqrt{1-z}}K\left(\sqrt{\frac{z}{z-1}}\right),\quad
\Pi(n,\sqrt{z})= \frac{1}{(1-n)\sqrt{1-z}}\Pi\left(\frac{n}{n-1},\sqrt{\frac{z}{z-1}}\right)
\end{equation*}
to obtain
\[
\dfrac{dh(k)}{dk}= \frac{1}{(k-4)\pi}\left(K\left(\frac{4}{k}\right)-\frac{8}{k}\Pi\left(-\frac{4}{k},\frac{4}{k}\right)\right).
\]
By Lemma~\ref{L:dfk}, one sees that \eqref{E:dfk2} is equivalent to 
\begin{equation}\label{E:ei}
\Pi\left(-\frac{4}{k},\frac{4}{k}\right)-\frac{1}{2}K\left(\frac{4}{k}\right)=\frac{k\pi}{4(k+4)} \quad \text{ for } k>4,
\end{equation}
which is merely a special case of Example~\ref{Ex:1} in Appendix~\ref{A:1}. 
\end{proof}
We finish this section by proving our main result.
\begin{proof}[Proof of Theorem~\ref{T:main} and Corollary~\ref{Cor:1}]
Let $\tilde{P}_k$, $B(x)$, and $y_{\pm}(x)$ be as defined in the proof of Lemma~\ref{L:main}. If $k>2(1+\sqrt{5})$ and $\theta\in \R$, then
\[B(e^{i\theta})= \frac{2\cos\theta\sqrt{k+4}-k}{\sqrt{k-4}}\le \frac{2\sqrt{k+4}-k}{\sqrt{k-4}}<0.\]
Since $y_+(x)y_-(x)=-1$, it follows that $|y_-(x)|<1<|y_+(x)|$ for $x$ lying on the unit circle, implying $\m^-(P_{a,c})=\m^-(\tilde{P}_k)=0$ and $h(k)=\m^+(P_{a,c})=\m(P_{a,c}).$ Hence Corollary~\ref{Cor:1} is an immediate consequence of Theorem~\ref{T:main}.

To prove Theorem~\ref{T:main}, we simply integrate both sides of \eqref{E:dfk2}. The result is as follows: for $k>4$
\begin{equation*}
f(k)=
 2h(k)+\frac{1}{2}\log\left(\mfrac{k-4}{k+4}\right)+C,
\end{equation*}
where $C$ is a constant. By the above argument, for $k>2(1+\sqrt{5})$, we have
\begin{align*}
f(k)&=\m\left(x+\frac{1}{k}+y+\frac{1}{y}+k\right)=\m\left(1+\frac{1}{k}\left(x+\frac{1}{k}+y+\frac{1}{y}\right)\right)+\log k,\\
h(k)&=\m(\tilde{P}_k)=\m\left(\sqrt{\frac{k+4}{k-4}}\left(x+\frac{1}{x}\right)+y-\frac{1}{y}-\frac{k}{\sqrt{k-4}}\right)\\
&=\m\left(\sqrt[4]{\frac{k+4}{k^2(k-4)}}\left(x+\frac{1}{x}\right)+\sqrt[4]{\frac{k-4}{k^2(k+4)}}\left(y-\frac{1}{y}\right)-\sqrt[4]{\frac{k^2}{k^2-16}}\right)\\
&\qquad +\log\left(\sqrt[4]{\frac{k^2(k+4)}{k-4}}\right).
\end{align*}
Taking $k\rightarrow \infty$ leads to
\[C=f(k)-2h(k)-\frac{1}{2}\log\left(\mfrac{k-4}{k+4}\right)\rightarrow 0.\]
Hence we conclude that for $k>4$
\[f(k)=
 2h(k)+\frac{1}{2}\log\left(\mfrac{k-4}{k+4}\right),\]
as desired.
\end{proof}
\section{Final remarks}
For $k\in\R$, we can relate $\m(P_{1,k})$ to half-Mahler measures of $P_{a,c},$ where $a$ and $c$ are suitable algebraic functions of $k$, thanks to Theorem~\ref{T:main} and \cite[Thm.~2]{LSZ}. As explained in the introduction, we are interested in finding this type of identity since it could be useful in verifying some conjectures of Boyd \cite{LSZ,MS}. These results rely on the existence of modular unit parametrizations of some elliptic curves in the non-tempered family $P_{a,c}$. We expected to obtain similar results for $\m(P_{1,k})$ with $k>4$. For example, choosing $k=5$, we find that the curve $P_{1,5}=0$ relates to $P_{3i,-5i}=0$ via Theorem~\ref{T:main} and the latter can obviously be transformed into
\begin{equation}\label{E:35}
\tilde{P}_5:=3\left(x+\frac{1}{x}\right)+y-\frac{1}{y}-5=0.
\end{equation}
By a classical result of Ramanujan \cite[Entry~62, p.~221]{Berndt}, this curve can be parametrized by the level $15$ modular functions 
\begin{align*}
x(\tau)&=-3\frac{\eta^2(3\tau)\eta^2(15\tau)}{\eta^2(\tau)\eta^2(5\tau)},\\
y(\tau)&=-\frac{\eta^3(\tau)\eta^3(15\tau)}{\eta^3(3\tau)\eta^3(5\tau)},
\end{align*}
where $\eta(\tau)$ is the usual Dedekind eta function. In other words, $\tilde{P}_5=0$ admits a modular unit parametrization, although no such parametrization exists for its cousin $P_{1,5}=0$. Using this fact together with a formula of Brunault, Mellit, and Zudilin \cite{Zudilin}, one should be able to give a direct proof of the following formula:
\[\m^+(\tilde{P}_{5})-\m^-(\tilde{P}_{5})=3L'(E_5,0)+\frac{1}{2}\log 3,\]
which is equivalent to one of Boyd's (proven) conjectures
\[\m(P_{1,5})=6L'(E_5,0).\]
However, we do not attempt to reprove this formula in this paper. 

It is also natural to ask whether there is an identity analogous to \eqref{E:main} for $k\not\in\R$. The case when $k$ is purely imaginary is of particular interest since there are several known (conjectural) formulas relating $\m(P_{1,ir})$, with $r\in\Z$, to $L$-values of elliptic curves. Despite an extensive search for such identity, we still have no positive answer to this question. It might be possible to detect new Mahler measure identities from a general elliptic integral identity which we examine in the appendix.

\appendix
\section{A special elliptic integral identity}\label{A:1}
Motivated by computational problems in particle physics, Jia \cite{Jia} discovered and proved the following intriguing identity
\begin{multline}\label{E:Jia}
\Pi\left(\frac{(1+x)(1-3x)}{(1-x)(1+3x)},\sqrt{\frac{(1+x)^3(1-3x)}{(1-x)^3(1+3x)}}\right)-\frac{1+3x}{6x}K\left(\sqrt{\frac{(1+x)^3(1-3x)}{(1-x)^3(1+3x)}}\right)\\= 
-\frac{\pi}{12}\frac{\sqrt{(1+3x)(x-1)^3}}{x} \quad \text{ for } x\in (-\infty,0) \cup (1,\infty).
\end{multline}
What makes this identity interesting is that, unlike most of the known identities, it involves only complete elliptic integrals of the first and the third kind and a simple algebraic function. Thus it can be used to replace the complicated $\Pi$ function with simpler functions in a favorable situation. It turns out that there are many other identities which look similar to \eqref{E:Jia}. In spite of their seemingly unrelated origins, some of them are equivalent to identities between Mahler measures in the family $P_{a,c}$, including the one we prove in this paper. We propose a generalization of this identity below.
\begin{theorem}\label{T:EE}
Let $p(x)$ and $q(x)$ be differentiable functions on an open set $A\subseteq \R$ and let 
\begin{align}
r(x)&=\frac{p'(x)q(x)(1-q(x)^2)+2q'(x)q(x)^2(p(x)-1)}{2q'(x)(1-p(x))(q(x)^2-p(x))},\label{E:r}\\
f(x)&= \frac{p'(x)(p(x)^2-q(x)^2)-2q'(x)q(x)p(x)(p(x)-1)}{2p(x)(p(x)-1)(q(x)^2-p(x))}.\label{E:f}
\end{align} 
If $r(x)$ is differentiable on $A$ and satisfies the following first-order nonhomogeneous differential equation
\begin{equation}\label{E:DE}
r'(x)-\left(f(x)+\frac{q'(x)}{q(x)}\right)r(x)= -\frac{p'(x)}{2p(x)(p(x)-1)}
\end{equation} 
and $\Pi(p(x),q(x))$ and $K(q(x))$ are differentiable on $A$,
then the following identity is valid for all $x\in A$
\begin{equation}\label{E:id}
\Pi(p(x),q(x))+r(x)K(q(x))=s(x),
\end{equation}
where $s(x)=e^{\int f(x)dx}$.
\end{theorem}
\begin{proof}
Let $w(x)=\Pi(p(x),q(x))+r(x)K(q(x)).$ Using the differentiation formulas for the elliptic integrals of the first and the third kind \cite[Ch.~19]{DLMF}, we find that
\begin{multline*}
w'(x) = \left(\frac{p'(x)}{2(p(x)-1)(q(x)^2-p(x))}+\frac{q'(x)q(x)}{(1-q(x)^2)(q(x)^2-p(x))}+\frac{r(x)q'(x)}{q(x)(1-q(x)^2)}\right)E(q(x))\\
+\left(\frac{p'(x)}{2p(x)(p(x)-1)}-\frac{r(x)q'(x)}{q(x)}+r'(x)\right)K(q(x))\\
+\left(\frac{p'(x)(p(x)^2-q(x)^2)}{2p(x)(p(x)-1)(q(x)^2-p(x))}-\frac{q'(x)q(x)}{q(x)^2-p(x)}\right)\Pi(p(x),q(x)).
\end{multline*}
By \eqref{E:r}, \eqref{E:f}, and simple manipulations, one sees that the coefficients of $E(q(x))$ and $\Pi(p(x),q(x))$ are $0$ and $f(x)$, respectively, whence
\begin{align*}
w'(x)&=f(x)\left(\Pi(p(x),q(x))+\frac{1}{f(x)}\left(\frac{p'(x)}{2p(x)(p(x)-1)}-\frac{r(x)q'(x)}{q(x)}+r'(x)\right)K(q(x))\right)\\
&= f(x)w(x),
\end{align*}
where the latter equality follows from \eqref{E:DE}. Finally, we deduce \eqref{E:id} by solving the above differential equation.
\end{proof}
\begin{remark}
It seems plausible to simplify the assumption in Theorem~\ref{T:EE} further to make it become more practical. For instance, using the integrating factor method, we can write a general solution to \eqref{E:DE} as 
\[r(x)= -\frac{1}{u(x)}\int\frac{u(x)p'(x)}{2p(x)(p(x)-1)}dx,\]
where $u(x)=\sqrt{\frac{(p(x)-1)(q(x)^2-p(x))}{p(x)q(x)^2}}$. However, it is unclear to us whether there is a `hidden' general relationship between $p(x)$ and $q(x)$ which enables us to obtain a closed-form of $r(x)$ from the above integral and is compatible with \eqref{E:r}. For any fixed function $q(x)$ with nice properties, one can try to find $p(x)$ which gives \eqref{E:id} by using a generic series expansion of $p(x)$ and solving for its coefficients via the differential equation \eqref{E:DE}. Although symbolic computation might be helpful in the process of finding suitable $p(x)$ and $q(x)$, we find Theorem~\ref{T:EE} useful for verifying, rather than finding, a new identity. As one can see from the examples below, once $p(x)$ and $q(x)$ are properly chosen, the condition \eqref{E:DE} can be easily checked in a standard computer algebra system. It would also be desirable to find a larger domain in the complex plain for which \eqref{E:id} is valid. Numerical evidence for \eqref{E:Jia} with $x\in \C$ is given at the end of \cite{Jia}. 
\end{remark}

\begin{example}\label{Ex:1}
Let $p(x)=-x$ and $q(x)=x$. Then \[r(x)=-\frac12,\quad f(x)=-\frac{1}{x+1}, \quad e^{\int f(x) dx}=\frac{1}{x+1},\]
and $r(x)$ obviously satisfies \eqref{E:DE}. Hence for $x\in (-1,1)$ we have
\[\Pi(-x,x)-\frac{1}{2}K(x)=\frac{C}{x+1},\]
for some constant $C$. Since $\Pi(0,0)=K(0)=\pi/2$, one sees immediately that $C=\pi/4.$ This identity implies \eqref{E:ei}, which is a crucial step in the proof of the main result of this paper.
\end{example}
\begin{example}
Let $p(x)=\frac{(1+x)(1-3x)}{(1-x)(1+3x)}$ and $q(x)=\sqrt{\frac{(1+x)^3(1-3x)}{(1-x)^3(1+3x)}}$. Then
\[r(x)=-\frac{1+3x}{6x},\quad f(x)=\frac{3}{2}\left(\frac{1}{x-1}+\frac{1}{1+3x}\right)-\frac{1}{x}, \quad e^{\int f(x) dx}=\frac{\sqrt{(1+3x)(x-1)^3}}{x},\]
and $r(x)$ satisfies \eqref{E:DE}. Hence for $x\in(-\infty,-1/3)$ we have Jia's identity
\[\Pi(p(x),q(x))-\frac{1+3x}{6x}K(q(x))=-\frac{\pi}{12}\frac{\sqrt{(1+3x)(x-1)^3}}{x},\]
where the constant $-\frac{\pi}{12}$ is obtained by choosing $x=-1$. The domain for $x$ can be extended to $(-\infty,0)\cup (1,\infty)$ using careful examination on the (complex) values of $\Pi(p(x),q(x))$ and $K(q(x))$. This identity is invoked in the proof of \cite[Lem.~3]{MS}.
\end{example}

\begin{example}
Let $p(x)=-\frac{x^2}{1+2x}$ and $q(x)=\sqrt{\frac{x^3(2+x)}{1+2x}}.$ Then 
\[r(x)=-\frac{(2+x)(1+2x)}{3(1+x)^2},\quad f(x)=\frac{1}{1+2x}-\frac{2}{1+x}, \quad e^{\int f(x) dx}=\frac{\sqrt{1+2x}}{(1+x)^2},\]
and $r(x)$ satisfies \eqref{E:DE}. Hence for $x\in(0,1)$ we have
\[\Pi(p(x),q(x))-\frac{1+3x}{6x}K(q(x))=\frac{\pi}{6}\frac{\sqrt{1+2x}}{(1+x)^2},\]
which is used in the proof of \cite[Lem.~6]{MS}.
\end{example}

\begin{example}
Let $p(x)=x(\sqrt{x^2+1}+1)(\sqrt{x^2+1}-x)$ and $q(x)=x^2$. Then 
\[r(x)=\frac{1-x-2\sqrt{1+x^2}}{4\sqrt{1+x^2}},\quad f(x)=\frac{1}{1-x}-\frac{x}{1+x^2}, \quad e^{\int f(x) dx}=\frac{1}{(1-x)\sqrt{1+x^2}},\]
and $r(x)$ satisfies \eqref{E:DE}. Hence for $x\in(0,1)$ we have
\[\Pi(p(x),q(x))+\frac{1-x-2\sqrt{1+x^2}}{4\sqrt{1+x^2}}K(q(x))=\frac{3\pi}{8(1-x)\sqrt{1+x^2}},\]
which appears in the proof of \cite[Lem.~6]{LSZ}.
\end{example}

\vspace{0.5 in}

\textbf{Acknowledgements}
This research is supported by the Thailand Research Fund (TRF) and the Office of the Higher Education Commission (OHEC) under the Research Grant for New Scholar MRG6280045. 

\bibliographystyle{amsplain}
\bibliography{Mahler}

\providecommand{\bysame}{\leavevmode\hbox to3em{\hrulefill}\thinspace}
\providecommand{\MR}{\relax\ifhmode\unskip\space\fi MR }
\providecommand{\MRhref}[2]{%
  \href{http://www.ams.org/mathscinet-getitem?mr=#1}{#2}
}
\providecommand{\href}[2]{#2}
\begin{thebibliography}{10}

\bibitem{Berndt}
Bruce~C. Berndt, \emph{Ramanujan's notebooks. {P}art {IV}}, Springer-Verlag,
  New York, 1994. \MR{1261634}

\bibitem{BZ}
Marie~Jos\'e Bertin and Wadim Zudilin, \emph{On the {M}ahler measure of a
  family of genus 2 curves}, Math. Z. \textbf{283} (2016), no.~3-4, 1185--1193.
  \MR{3519999}

\bibitem{Boyd}
David~W. Boyd, \emph{Mahler's measure and special values of {$L$}-functions},
  Experiment. Math. \textbf{7} (1998), no.~1, 37--82. \MR{1618282}

\bibitem{Brunault}
Fran\c{c}ois Brunault, \emph{Regulators of {S}iegel units and applications}, J.
  Number Theory \textbf{163} (2016), 542--569. \MR{3459587}

\bibitem{DLMF}
\emph{{\it NIST Digital Library of Mathematical Functions}},
  http://dlmf.nist.gov/, Release 1.0.23 of 2019-06-15, F.~W.~J. Olver, A.~B.
  {Olde Daalhuis}, D.~W. Lozier, B.~I. Schneider, R.~F. Boisvert, C.~W. Clark,
  B.~R. Miller and B.~V. Saunders, eds.

\bibitem{Giard}
Antoine Giard, \emph{Mahler measure of a non-tempered {W}eierstrass form}, J.
  Number Theory \textbf{209} (2020), 225--245. \MR{4053066}

\bibitem{Hall}
Leon~M. Hall, \emph{Missouri {S}\&{T} {M}ath 483, {L}ecture {N}otes: {S}pecial
  {F}unctions}.

\bibitem{Jia}
Yu~{Jia}, \emph{{One Special Identity between the complete elliptic integrals
  of the first and the third kind}}, arXiv e-prints (2008), arXiv:0802.3977.

\bibitem{LSZ}
Matilde Lal\'{i}n, Detchat Samart, and Wadim Zudilin, \emph{Further
  explorations of {B}oyd's conjectures and a conductor 21 elliptic curve}, J.
  Lond. Math. Soc. (2) \textbf{93} (2016), no.~2, 341--360. \MR{3483117}

\bibitem{Lalin10}
Matilde~N. Lal\'{\i}n, \emph{On a conjecture by {B}oyd}, Int. J. Number Theory
  \textbf{6} (2010), no.~3, 705--711. \MR{2652904}

\bibitem{LR}
Matilde~N. Lalin and Mathew~D. Rogers, \emph{Functional equations for {M}ahler
  measures of genus-one curves}, Algebra Number Theory \textbf{1} (2007),
  no.~1, 87--117. \MR{2336636}

\bibitem{MS}
Yotsanan Meemark and Detchat Samart, \emph{Mahler measures of a family of
  non-tempered polynomials and {B}oyd's conjectures}, Res. Math. Sci.
  \textbf{7} (2020), no.~1, Paper No. 1, 20. \MR{4042306}

\bibitem{Mellit}
Anton Mellit, \emph{Mahler measures and $q$-series}, Explicit Methods in Number
  Theory (MFO, Oberwolfach, Germany, 17--23 July 2011), Oberwolfach Reports,
  vol.~8, 2011, pp.~1990--1991.

\bibitem{RZ}
Mathew Rogers and Wadim Zudilin, \emph{From {$L$}-series of elliptic curves to
  {M}ahler measures}, Compos. Math. \textbf{148} (2012), no.~2, 385--414.
  \MR{2904192}

\bibitem{RZ15}
\bysame, \emph{On the {M}ahler measure of {$1+X+1/X+Y+1/Y$}}, Int. Math. Res.
  Not. IMRN (2014), no.~9, 2305--2326. \MR{3207368}

\bibitem{RV}
F.~Rodriguez Villegas, \emph{Modular {M}ahler measures. {I}}, Topics in number
  theory ({U}niversity {P}ark, {PA}, 1997), Math. Appl., vol. 467, Kluwer Acad.
  Publ., Dordrecht, 1999, pp.~17--48. \MR{1691309}

\bibitem{Zudilin}
Wadim Zudilin, \emph{Regulator of modular units and {M}ahler measures}, Math.
  Proc. Cambridge Philos. Soc. \textbf{156} (2014), no.~2, 313--326.
  \MR{3177872}

\end{thebibliography}
\end{document}